\documentclass{amsart}
\usepackage{amsmath,amssymb,graphicx,amsthm,bm,hyperref,mathrsfs,tikz}
\usepackage{graphicx}
\usepackage{hyperref}
\hypersetup{colorlinks=true,citecolor=blue,linkcolor=cyan}
\newtheorem{thm}{Theorem}[section]

\newtheorem{lem}[thm]{Lemma}

\theoremstyle{definition}
\newtheorem{defn}[thm]{Definition}
\theoremstyle{remark}

\theoremstyle{question}
\newtheorem{question}[thm]{Question}
\numberwithin{equation}{section}

\newcommand{\Ff}{\mathbb{F}}
\newcommand{\Z}{\mathbb{Z}}

\newcommand{\M}{\mathcal{M}}
\newcommand{\Ll}{\mathscr{L}}
\newcommand{\A}{\mathcal{A}}

\title[Multiplication Table Problem]{Erd\H{o}s' Multiplication Table Problem for Function Fields and Symmetric Groups}
\author{Patrick Meisner}
\address{School of Mathematical Science, Tel Aviv University, Ramat Aviv, Tel Aviv, 6997801, Israel }
\email{meisner@mail.tau.ac.il}

\begin{document}

\begin{abstract}

Erd\H{o}s first showed that the number of positive integers up to $x$ which can be written as a product of two number less than $\sqrt{x}$ has zero density. Ford then found the correct order of growth of the set of all these integers. We will use the tools developed by Ford to answer the analogous question in the function field setting. Finally, we will use a classical result relating factorization of polynomials to factorization of permutations to recover a result of Eberhard, Ford and Green of an analogous multiplication table problem for permutations.

\end{abstract}

\maketitle

\section{Introduction}\label{Intro}

Let $A(x)$ be the set of positive integers up to $x$ that can be written as a product of two numbers less than $\sqrt{x}$. Using estimates on the number of integers with a given number of prime divisors Erd\H{o}s \cite{E} was able to show that
$$|A(x)| \ll \frac{x}{(\log x)^{\delta} (\log\log x)^{1/2}},$$
where
$$\delta = 1 - \frac{1+\log\log2}{\log2} = 0.086071....$$
Much later, Ford \cite{F1,F2} considered the set $H(x,y,z)$ consisting of the number of integers up to $x$ which has a divisor in $(y,z]$. In particular, he showed that
\begin{align}\label{Hxy2y}
|H(x,y,2y)| \asymp \frac{x}{(\log y)^\delta (\log\log y)^{3/2}} \quad \quad (3 \leq y \leq \sqrt{x}),
\end{align}
and that
$$\left|H\left(\frac{x}{4}, \frac{\sqrt{x}}{4}, \frac{\sqrt{x}}{2}\right)\right| \leq |A(x)| \leq \sum_{k\geq 0} \left|H\left(\frac{x}{2^k}, \frac{\sqrt{x}}{2^{k+1}}, \frac{\sqrt{x}}{2^k}\right)\right|$$
from which you can then conclude that
\begin{align}\label{A(x)}
|A(x)| \asymp  \frac{x}{(\log x)^\delta (\log\log x)^{3/2}}.
\end{align}
Here we use the notation that $f(x) \ll g(x)$ if there is a constant $C$ and $X>0$ such that $|f(x)|\leq |g(x)|$ for all $x\geq X$. Further, we write $f(x) \asymp g(x)$ to mean $f(x) \ll g(x)$ and $g(x)\ll f(x)$.

Several authors have generalized this problem to various other settings. Koukoulopoulos \cite{K1,K2,K3} considered the number of integers up to $x$ that can be written as a product of $k$ different integers in certain intervals, the so-called generalized multiplication table problem. Eberhard, Ford and Green \cite{EFG} considered an analogous problem for permutations in the symmetric group (see Section \ref{symgroup} for further discussion) while the first two authors with Koukoulopolous \cite{EFK} looked at the generalized multiplication table problem for the symmetric group. Finally, Mangerel proved the analogous statement for arithmetical semigroups that satisfy an ``$\alpha$-prime element theorem" (see \cite{M} for more details). We are interested, however, in the analogous statement in the function field setting.

\subsection{Function Field Analogy}

There is a dictionary of sorts that relates statements made about integers to statements about polynomials over finite fields:
\begin{center}
\begin{tabular} { c| c}
$\Z$ & $\Ff_q[t]$ \\
\hline
$p$, prime & $P$, prime polynomial \\
positive & monic \\
$|m|$ & $|F|=q^{\deg F}$ \\
$m\leq x$ & $\deg F=n$ \\
$\log x$ & $\deg F$
\end{tabular}
\end{center}

Therefore, we can make statements about function fields that are analogous to statements in the integers by replacing the appropriate ``words". For example the prime number theorem states that the number of primes up to $x$ is
$$\pi(x) := |\{p \leq x : p \mbox{ is prime}\}| \sim \frac{x}{\log(x)}.$$
The analogous question, the prime polynomial theorem, asks how many prime polynomials over $\Ff_q[t]$ are there of degree $n$ with the answer being
$$\pi_q(n) := |\{P\in \M_n : P \mbox{ is a prime polynomial}\}| = \frac{q^n}{n} + O\left(\frac{q^{n/2}}{n}\right),$$
where $\M_n$ is the set of monic polynomials of degree $n$. Note that we get a square-root saving in the error term for the prime polynomial theorem as the Riemann Hypothesis is known for function fields.

Using this dictionary we can create analogous sets to $A(x)$ and $H(x,y,2y)$ in the function field setting. The background set for $A(x)$ is all the positive integers less than $x$ so the background set in the function field setting would be $\M_n$, the monic polynomials of degree $n$. Since degree is the analogy of $\log$, the condition of being the product of two integers less than $\sqrt{x}$ in $A(x)$ is analogous to being the product of two polynomials of degree $n/2$. Clearly this only makes sense if $n$ is even and so we define
\begin{align}\label{M2n}
M(2n) := \{ F\in\ \M_{2n} : F=GH, G,H\in \M_n \}.
\end{align}
Then the multiplication table problem would be to find the size of set $M(2n)$. Using the dictionary we can make a good guess as to how large the set should be and in fact that is what we get.

\begin{thm}\label{MainThm1}
$$|M(2n)|\asymp \frac{q^{2n}}{n^\delta (1+\log n)^{3/2}}$$
as $q^n\to\infty$.
\end{thm}
Notice that since $n$ replaces $\log x$, $\log n$ replaces $\log\log x$. Moreover, we have a $(1+\log(n))^{3/2}$ in the denominator to correct for when $n=1$.

The analogy for $H(x,y,2y)$ is a little subtler. We must ask ourselves what is the correct analogue of $2$ and the importance it plays in the proof of \eqref{Hxy2y}. In fact $2$ is important in this context because it is the smallest prime. While the concept of a smallest prime is not well defined in the function field setting, the degree of the smallest prime is well defined, and it is $1$. Therefore, the analogue of a number having a divisor in $(y,2y]$ would be for a polynomial to have a divisor with degree in $(b,b+1]$. But since the degree is always an integer we see that this is equivalent to saying a polynomial has a divisor of some fixed degree. Thus we define
\begin{align}\label{Hnb}
H(n,b) := & \{F \in \M_n: F \mbox{ has a divisor of degree } b\} \\
 = & \{F\in \M_n : F=GH, G\in \M_b, H\in \M_{n-b}\}. \nonumber
\end{align}
Moreover, we see that $H(n,b)=H(n,n-b)$ so we may always assume that $b\leq n/2$. Again, the result predicted by the dictionary is the truth.

\begin{thm}\label{MainThm2}
For $b\leq n/2$,
$$|H(n,b)| \asymp \frac{q^{n}}{b^\delta (1+\log b)^{3/2}}$$
as $q^n\to\infty$.
\end{thm}

Of course $M(2n)=H(2n,n)$ so Theorem \ref{MainThm1} is a direct corollary of Theorem \ref{MainThm2}.

\subsection{Symmetric Groups}\label{symgroup}

Let $S_n$ be the symmetric group on $n$ elements and define
$$I(n,b) := \{\sigma \in S_n : \sigma \mbox{ fixes some subset of size } b\}.$$

Eberhard, Ford and Green \cite{EFG} adapted the methods of Ford in \cite{F1,F2} to show that
\begin{thm}\label{EFG}
For $b\leq n/2$,
$$|I(n,b)| \asymp \frac{n!}{b^\delta(1+\log b)^{3/2}}$$
as $n\to\infty$.
\end{thm}

As well as the analogy between integers and polynomials over a finite field, there is an analogy between polynomials over a finite field of degree $n$ and the symmetric group on $n$ elements. In particular, one can show that, in the $q$-limit, the probability a polynomial has a given factorization into prime polynomials is the same as the probability a permutation has the same factorization type into cyclic elements. Through this analogy we can relate the relative size of $I(n,b)$ to the relative size of $H(n,b)$.

\begin{thm}\label{RelThm}

$$\lim_{q\to\infty} \frac{|H(n,b)|}{q^n} = \frac{|I(n,b)|}{n!}.$$

\end{thm}

The proof of Theorem \ref{RelThm} is independent of Theorems \ref{MainThm2} and \ref{EFG}. Therefore Theorems \ref{EFG} and \ref{RelThm} imply Theorem \ref{MainThm2} for $n$ fixed and $q$ tending to infinity. However, the proof we give here of Theorem \ref{MainThm2} is independent of Theorem \ref{EFG} and is valid for $q^n$ tending to infinity in any way (in particular, for $q$ fixed and $n$ tending to infinity). Hence we get a new proof of Theorem \ref{EFG}.

Define these two properties of permutations on $S_n$:
\begin{defn}

We say $\sigma,\tau\in S_n$ are \textbf{disjoint} if they permute different elements. That is, if $\sigma(k)\not=k$ then $\tau(k)=k$ and, vice versa, if $\tau(k)\not=k$ then $\sigma(k)=k$.

\end{defn}

\begin{defn}

We say $\sigma\in S_n$ \textbf{embeds into $S_m$} if there is a subset $I\subset \{1,\dots,n\}$ of size $m$ such that $\sigma$ permutes $I$ and is trivial outside of $I$. That is, $\sigma(k)\in I$ for all $k\in I$ and $\sigma(k)=k$ for all $k\not \in I$.

\end{defn}

Then we see that $I(n,b)$ has an equivalent definition
\begin{align}\label{Inb}
I(n,b) := \{\sigma \in S_n : \sigma = \tau_1\tau_2 \mbox{ such that $\tau_1,\tau_2$ are disjoint and $\tau_1$ embeds into $S_b$}\}.
\end{align}

In this way we see that $I(2n,n)$ is a reasonable analogue of the multiplication table set in $S_{2n}$. However, Theorem \ref{EFG} is then surprising as one would expect from \eqref{A(x)} and Theorem \ref{MainThm2} that the multiplication table set of $S_{2n}$ would have size roughly
\begin{align*}
\frac{|S_{2n}|}{(\log |S_{2n}|)^{\delta} (1+\log\log |S_{2n}|)^{3/2}} & = \frac{(2n)!}{(\log((2n)!))^{\delta} (1+\log\log((2n)!))^{3/2}} \\
&  \asymp \frac{(2n)!}{(n\log n)^{\delta} (1+\log n)^{3/2}}
\end{align*}
So this raises the question:

\begin{question}
Does there exist a different (more reasonable) analogue of the multiplication table set in $S_{2n}$ that has size roughly like the above equation?
\end{question}

\textbf{Outline of the paper:} In Section \ref{Section2} we will prove Theorem \ref{RelThm}. Then Sections \ref{Section3} and \ref{Section4} will be devoted to proving the lower and upper bounds for Theorem \ref{MainThm2}, respectively. We will use the techniques developed by Ford to reduce the question down to the same estimates as for the integer case. Finally, we include an appendix with proofs of function field analogues of well known useful results in the integer setting.

We will preserve the variable $P$ (with any subscript) to denote a prime polynomial. Moreover, for brevity, if we write a sum (or product) with $P$ in the subscript, we will always have this denote the sum (or product) over prime polynomials that satisfy the other conditions imposed by the sum (or product).

\textbf{Acknowledgements:} I would like to thank Ben Green, Dimitris Koukoulopoulos and Sacha Mangerel for bringing to my attention some of the literature of the field. I would also like to thank Ofir Gorodetsky for pointing out some inaccuracies in the appendix in an earlier version.

The research leading to these results has received funding from the European Research Council under the European Union's Seventh Framework Programme (FP7/2007-2013) / ERC grant agreement n$^{\text{o}}$ 320755.

\section{Symmetric Groups}\label{Section2}

Let $F\in \M_n$. Suppose it can be factored as $F = \prod_{i=1}^t P_i$ where the $P_i$ are not necessarily distinct primes. Then the tuple $(\deg P_1,\dots,\deg P_t)$ gives a partition of $n$. Denote this partition as $\lambda_F$. Further, for any partition $\lambda$ of $n$, define
$$\pi_q(n,\lambda) = |\{F\in \M_n : \lambda_F=\lambda\}|$$
to be the number of polynomials of degree $n$ with a fixed factorization type. Note that if we set $\lambda = (n)$, the partition consisting only of $n$, then we see that $\pi_q(n,(n))=\pi_q(n)$, the number of primes of degree $n$.

Likewise, all $\sigma\in S_n$ can be decomposed as $\sigma = \prod_{i=1}^t c_i$ where the $c_i$ are disjoint cycles. Then the tuple $(\ell(c_1),\dots,\ell(c_t))$ gives a partition of $n$, where $\ell(c_i)$ is the length of $c_i$. Denote this partition $\lambda_\sigma$. Note: if $\sigma(k)=k$, then we include the cycle $(k)$ in the decomposition of $\sigma$ and this contributes a $1$ to the partition of $n$. Now, for any partition $\lambda$ of $n$, define
$$P(\lambda) = \frac{|\{\sigma\in S_n : \lambda_\sigma = \lambda\}|}{n!}$$
to be the probability that a permutation has a certain cycle decomposition. Then there is a classical result that follows directly from the prime polynomial theorem:

\begin{thm}\label{BBR}[Lemma 2.1 of \cite{ABR}]

Let $n$ be a positive integer. Then there exists a $c(n)>0$ depending only on $n$ such that
$$|\pi_q(n,\lambda) - P(\lambda)q^n| \leq c(n) q^{n-1}.$$

\end{thm}

We can now use this result to prove Theorem \ref{RelThm}.

\begin{proof}[Proof of Theorem \ref{RelThm}]

We will say $\lambda$ has a $b$-subpartition if there exists a subset of $\lambda$ that is a partition of $b$. Therefore $F\in H(n,b)$ if and only if $\lambda_F$ has a $b$-subpartition. Indeed if $F\in H(n,b)$ then $F=GH$ with $G\in\M_b$ and $\lambda_G$ will be a $b$-subpartition of $\lambda_F$. Conversely, if $\lambda'$ is a $b$-subpartition of $\lambda_F$, then define $G$ to be the product of the primes of $F$ corresponding to $\lambda'$. Then $G|F$ and $G\in\M_b$ and hence $F\in H(n,b)$.

Let
$$\Lambda(n,b) := \{\lambda: \lambda \mbox{ is a partition of $n$ with a $b$-subpartition} \}.$$
Then we get that
$$H(n,b) = \bigcup_{\lambda\in\Lambda(n,b)} \{F\in \M_n : \lambda_F=\lambda\}.$$
Moreover, this union is disjoint as if $\lambda_F\not=\lambda_G$ then $F\not=G$. Therefore,
\begin{align*}
|H(n,b)| & = \sum_{\lambda\in\Lambda(n,b)}|\{F\in M_n: \lambda_F = \lambda\}|\\
& = \frac{q^n}{n!}\sum_{\lambda\in\Lambda(n,b)}| \{\sigma\in S_n : \lambda_\sigma=\lambda\}| + O\left(c(n)q^{n-1}e^{\pi \sqrt{2n/3}}\right)
\end{align*}
where the last equality comes from Theorem \ref{BBR} and bounds on the number of partition of $n$ as proved by Hardy and Ramanujan \cite{HR}.

Furthermore, $\sigma\in I(n,b)$ if and only if $\lambda_\sigma\in \Lambda(n,b)$. Indeed if $\sigma\in I(n,b)$ then, using the second definition of $I(n,b)$ in the introduction, $\sigma=\tau_1\tau_2$ with $\tau_1$ and $\tau_2$ disjoint and $\tau_1$ embeds into $S_b$ therefore, $\lambda_{\tau_1}$ will be a $b$-subpartition of $\lambda_\sigma$. Conversely, if $\lambda_\sigma$ has a $b$-subpartition then let $\tau_1$ be the product of the cycles corresponding to the subpartition and $\tau_2$ be the product of the remaining cycles. Then $\tau_1$ will embed into $S_b$, $\tau_1$ and $\tau_2$ will be disjoint and $\sigma=\tau_1\tau_2$.

Therefore, we get
$$I(n,b) = \bigcup_{\lambda\in\Lambda(n,b)} \{\sigma\in S_n: \lambda_\sigma=\lambda\}$$
and since this union is disjoint (as $\lambda_\sigma\not=\lambda_\tau$ implies $\sigma\not=\tau$) then we finally have
\begin{align*}
\frac{|I(n,b)|}{n!} & = \sum_{\lambda\in\Lambda(n,b)}| \{\sigma\in S_n : \lambda_\sigma=\lambda\}| \\
& = \frac{|H(n,b)|}{q^n} +  O\left(\frac{c(n)e^{\pi \sqrt{2n/3}}}{q}\right).
\end{align*}
Finally, if we let $q$ tend to infinity, then the big-$O$ term will go to zero.

\end{proof}

\section{Lower Bound}\label{Section3}

In Ford's proof for the integers, he expresses the size of $H(x,y,2y)$ in terms of ``a measure of the degree of clustering of the divisors of an integer $a$" which he defines as
$$L(a) = \mbox{meas}\Ll(a), \quad \quad \quad \Ll(a) = \bigcup_{d|a}[\log(d/2), \log d).$$
Again, here the importance of $2$ is just that it is the smallest prime integer. The analogue of  $\log 2$ in the function field setting is then just the degree of the smallest prime, which is $1$. Hence, for a polynomial $A$ and a divisor $D$ of $A$, the corresponding interval we will want to consider is something of the form $[\deg(D)-1,\deg(D))$. However, since the $\deg$ function only takes integer values, we actually only care about the singleton $\deg(D)$. Hence, we will define
$$\Ll(A) = \{d : d=\deg(D) \mbox{ for some } D|A\}, \quad \quad \quad L(A)=|\Ll(A)|.$$

\begin{lem}\label{lowboundlem1}
For $b\leq n/2$,
$$|H(n,b)| \gg \frac{q^n}{b^2} \sum_{\deg(A)\leq b/8} \frac{L(A)}{|A|}$$
as $q^n\to\infty$.
\end{lem}

\begin{proof}

Consider the set of polynomials of the form $F=APB$ where $\deg(A) \leq b/8$, $P$ is a prime with $b-\deg(P)\in\Ll(A)$ and all the primes of $B$ have degree $\geq b$ or in $[b/4,3b/4]$. The condition on $P$ implies that $AP$ has a divisor of degree $b$. Moreover, we have $7b/8\leq \deg(P) \leq b$ and so every polynomial of this form has a unique representation. Fix $A,P$ and note that $\deg(B)=n-\deg(AP)\geq 7b/8$. If $\deg(B)\geq b$ then, by \eqref{smooth}, the number of such $B$ will be greater than
$$|\{B\in \M_{n-\deg(AP)} : \deg(P^-(B))\geq b\}| \asymp \frac{q^{n-\deg(AP)}}{b} = \frac{q^n}{b|AP|},$$
where $P^-(B)$ is the smallest prime divisor of $B$.

Otherwise, if $\deg(B)<b$, then $B$ will have at least two prime divisors from $[b/4,3b/4]$. Hence the number of such $B$ will be greater than
\begin{align*}
\sum_{\substack{d_1,d_2\in [b/4,3b/4] \\ d_1+d_2 = n-\deg(AP)}} \pi_q(d_1)\pi_q(d_2) & = \sum_{d = b/4}^{n-b/4-\deg(AP)} \pi_q(d)\pi_q(n-\deg(AP)-d) \\
& \gg \frac{q^{n}}{|AP|}\sum_{d = b/4}^{5b/4} \frac{1}{d(b-d)} \gg \frac{q^{n}}{b|AP|},
\end{align*}
where $\pi_q(n)$ is the number of prime polynomials of degree $n$.

Therefore,
$$|H(n,b)| \geq \sum_{\deg(A)\leq b/8} \sum_{\substack{P \\ b-\deg(P)\in \Ll(A)}} \sideset{}{^*}\sum_B 1 \gg \frac{q^n}{b} \sum_{\deg(A)\leq b/8} \frac{1}{|A|} \sum_{\substack{P \\ b-\deg(P)\in \Ll(A)}}\frac{1}{|P|},$$
where $\sideset{}{^*}\sum$ indicates we sum over all such $B$ described above.

Finally,
$$\sum_{\substack{P \\ b-\deg(P)\in \Ll(A)}}\frac{1}{|P|} = \sum_{\substack{d \\ b-d\in \Ll(A)}} \frac{\pi_q(d)}{q^d} \sim \sum_{\substack{d \\ b-d\in \Ll(A)}} \frac{1}{d} \gg \frac{L(A)}{b}$$
and this completes the proof.

\end{proof}

For any polynomial $A$, let $\tau(A)$ be the number of divisors of $A$ and $\tau_d(A)$ be the number of divisors of $A$ of degree $d$. Then we clearly have
$$\tau(A) = \sum_{d\in \Ll(A)} \tau_d(A).$$
Then, for any subset $\A$ of polynomials we have by Cauchy-Schwarz that
\begin{align*}
\left(\sum_{A\in \A} \frac{\tau(A)}{|A|} \right)^2 & = \left(\sum_{A\in\A} \sum_{d\in\Ll(A)} \frac{\tau_d(A)}{|A|}\right)^2 \leq \left(\sum_{A\in\A}\sum_{d\in \Ll(A)} \frac{1}{|A|}\right) \left( \sum_{A\in\A}\sum_{d\in\Ll(A)} \frac{\tau_d(A)^2}{|A|}\right)\\
& = \left(\sum_{A\in\A} \frac{L(A)}{|A|}\right)\left(\sum_{A\in\A} \frac{W(A)}{|A|}\right),
\end{align*}
where
$$W(A) = \sum_{d\in \Ll(A)}\tau_d^2(A) = |\{(D,D') : D,D'|A, \deg(D)=\deg(D')\}|.$$
Hence if we have any collection of disjoint sets of polynomials $\A_1,\dots,\A_t$, all of whose degrees are less than $b/8$ then we get from Lemma \ref{lowboundlem1} that
\begin{align} \label{lowboundeq1}
|H(n,b)| \gg \frac{q^n}{b^2} \sum_{j=1}^t\sum_{A\in\A_j} \frac{L(A)}{|A|} \geq \frac{q^n}{b^2}\sum_{j=1}^t \frac{\left(\sum_{A\in\A_j} \frac{\tau(A)}{|A|}\right)^2}{\sum_{A\in \A_j} \frac{W(A)}{|A|}}
\end{align}

We will now construct appropriate sets that will give us the lower bound we desire. Towards this, partition the primes into subsets $D_1,D_2,\dots,$ such that $D_j$ consists of primes whose degree are in the interval $(\lambda_{j-1},\lambda_j]$ so that $\lambda_j$ is largest so that
$$\sum_{\deg(P)\in (\lambda_{j-1},\lambda_j]} \frac{1}{|P|} \leq \log(2).$$
Such partitions exists as a consequence of \eqref{inverseprime}. In fact, \eqref{inverseprime} tells us that for any $\lambda_{j-1}<\lambda_j$, we have
$$\sum_{\deg(P)\in (\lambda_{j-1},\lambda_j]} \frac{1}{|P|} = \log(\lambda_j) - \log(\lambda_{j-1}) +O\left(\frac{1}{\lambda_{j-1}}\right).$$
Therefore, there exists some constant $K$ such that
$$2^{j-K} \leq \lambda_j \leq 2^{j+K}.$$
Finally, for any $\bf{b}=(b_1,\dots,b_J)$ let $\A(\bf{b})$ be the set of square-free polynomials with exactly $b_j$ prime factors coming from the interval $D_j$.

\begin{lem}\label{WAlem}

$$\sum_{A\in \A(\bf{b})} \frac{W(A)}{|A|} \ll \frac{(2\log(2))^{b_1+\dots+b_J}}{b_1!\cdots b_J!} \sum_{j=1}^J  2^{-j+b_1+\cdots+b_j}.$$

\end{lem}

\begin{proof}

Let $B = b_1+\cdots+b_J$ and $A=P_1\cdots P_B$ such that
\begin{align}\label{intervalcond}
P_1,\dots,P_{b_1}\in D_1, P_{b_1+1},\dots,P_{b_1+b_2}\in D_2 \mbox{ and so on.}
\end{align}
Then $W(A)$ is the number of subsets $Y,Z \subset \{1,\dots,B\}$ such that
\begin{align}\label{YZcond}
\sum_{i\in Y} \deg(P_i) = \sum_{i\in Z} \deg(P_i).
\end{align}
Hence,
\begin{align}\label{WAbound1}
\sum_{A\in \A(\bf{b})} \frac{W(A)}{|A|} \leq \frac{1}{b_1!\cdots b_J!} \sum_{Y,Z \subset \{1,\dots,B\}} \sideset{}{'}\sum_{P_1,\dots,P_B} \frac{1}{|P_1|\cdots|P_B|}.
\end{align}
where $\sideset{}{'}\sum$ indicates that we are summing over all tuples $P_1,\dots,P_B$ that satisfy \eqref{intervalcond} and \eqref{YZcond}.

Consider the diagonal terms when $Y=Z$ of \eqref{WAbound1}:
$$\sum_{Y=Z \subset\{1,\dots,B\}}\sideset{}{'}\sum_{P_1,\dots,P_B} \frac{1}{|P_1|\cdots|P_B|} \leq \sum_{Y\subset\{1,\dots,B\}} \prod_{j=1}^J \left(\sum_{P_j\in D_j} \frac{1}{|P_j|}\right)^{b_j} \leq (2\log(2))^B $$

For the off-diagonal terms when $Y\not=Z$, let $I$ be the maximum element of $(Y\cup Z) \setminus (Y\cap Z)$. If we fix all the other $P_i$, then this fixes the degree of $P_I$ by \eqref{YZcond}. Moreover, if we let $E(I)$ be such that $P_I\in D_{E(I)}$ then $\deg(P_I)\geq \lambda_{E(I)-1} \gg 2^{E(I)}$. Therefore,
$$\sum_{P_I} \frac{1}{P_I} = \frac{\pi_q(\deg(P_I))}{q^{\deg(P_I)}}\ll \frac{1}{|\deg(P_I)|} \ll 2^{-E(I)}$$
Hence for a fixed $Y\not=Z$ we get
$$\sideset{}{'}\sum_{P_1,\dots,P_B} \frac{1}{|P_1|\cdots |P_B|} \ll 2^{-E(I)}\log(2)^B. $$

Finally, there are $2^{B+I-1}$ pairs of $Y,Z$ for each fixed $I$ and we get,
\begin{align*}
\sum_{A\in \A(b)} \frac{W(A)}{|A|} & \leq \frac{(2\log(2))^B}{b_1!\cdots b_J!} \left[ 1+\sum_{I=1}^B 2^{-E(I)}2^{I-1}  \right]\\
& \ll \frac{(2\log(2))^B}{b_1!\cdots b_J!} \sum_{j=1}^J 2^{-j}\sum_{I: E(I)=j}2^{I} \\
& \ll \frac{(2\log(2))^{b_1+\dots+b_J}}{b_1!\cdots b_J!} \sum_{j=1}^J 2^{-j} 2^{b_1+\cdots+b_j},
\end{align*}
where the last inequality follows from that fact that $E(I)=j$ if and only if $b_1+\cdots+b_{j-1}\leq I \leq b_1+\cdots b_j$.
\end{proof}

\begin{lem}\label{tauAlem}

If we suppose that $b_i=0$ for $i<M$ and $b_j\leq Mj$ for a sufficiently large $M$, then
$$\sum_{A\in \A(\textbf{b})} \frac{\tau(A)}{|A|} \gg \frac{(2\log(2))^{b_M+\cdots+b_J}}{b_M!\cdots b_J!}$$

\end{lem}

\begin{proof}

We have
$$\sum_{A\in \A(\textbf{b})} \frac{\tau(A)}{|A|} = 2^{b_M+\cdots+b_J} \prod_{j=M}^J \frac{1}{b_j!} \left(\sum_{\substack{P_1,\cdots,P_{b_j} \in D_j \\ P_i \mbox{ distinct}}} \frac{1}{|P_1\cdots P_{b_j}|} \right)$$
By the choice of $D_j$ and the prime polynomial theorem, we get that there is an absolute constant $C$ such that
$$\sum_{P\in D_j} \frac{1}{|P|} \geq \log(2) - \sum_{\deg(P)=\lambda_j+1} \frac{1}{|P|} \geq \log(2) - \frac{1}{\lambda_j+1} - \frac{C}{q^{\lambda_j/2}}.$$
Now, fix $P_1,\dots,P_k \in D_j$ and consider the sum
\begin{align*}
\sum_{\substack{P\in D_j \\ P\not=P_1,\dots,P_k}} \frac{1}{|P|} & = \sum_{P\in D_j} \frac{1}{|P|} - \sum_{i=1}^k \frac{1}{|P_i|} \geq \log(2) - \frac{1}{\lambda_j+1}- \frac{C}{q^{\lambda_j/2}} - \frac{k}{q^{\lambda_{j-1}}}
\end{align*}

Therefore,
\begin{align*}
\prod_{j=M}^J \left(\sum_{\substack{P_1,\cdots,P_{b_j} \in D_j \\ P_i \mbox{ distinct}}} \frac{1}{|P_1\cdots P_{b_j}|} \right) & \geq \prod_{j=M}^J \left(\log(2)-\frac{1}{\lambda_j+1}- \frac{C}{q^{\lambda_j/2}} - \frac{b_j}{q^{\lambda_{j-1}}}  \right)^{b_j} \\
& \geq \log(2)^{b_M+\cdots+b_J} \prod_{j=M}^J \left(1 - \frac{1}{\log(2)}\left(\frac{1}{\lambda_j+1}+ \frac{C}{q^{\lambda_j/2}} + \frac{b_j}{q^{\lambda_{j-1}}} \right) \right)^{b_j}.
\end{align*}
So it remains to show that this remaining product is bounded above. Indeed if we denote
$$C_j := \frac{1}{\log(2)}\left(\frac{1}{\lambda_j+1}+ \frac{C}{q^{\lambda_j/2}} + \frac{b_j}{q^{\lambda_{j-1}}} \right) \ll \frac{1}{2^j}$$
then
\begin{align*}
-\log \prod_{j=M}^J \left(1 - C_j\right)^{b_j} & = -\sum_{j=M}^J b_j \log \left(1-C_j\right) \\
& = \sum_{j=M}^J b_j\sum_{n=1}^{\infty} \frac{C_j^n}{n} \\
& \ll \sum_{j=M}^J  \sum_{n=1}^{\infty} \frac{j}{2^{nj}n} = O(1).
\end{align*}
This completes the proof.

\end{proof}

Finally, set $k = \lfloor \log_2(b) - 2M\rfloor$ and let $\mathcal{B}$ be the set of $\mathbf{b} = (b_1,\dots,b_J)$ with $J=M+k-1$, $b_j=0$ for $j\leq M$, $b_j \leq \min(Mj, M(J-j+1))$. Then for every $A\in \A(\mathbf{b})$, we have
\begin{align*}
\deg(A) & \leq \sum_{j=M}^J b_j \lambda_j \leq  M \sum_{\ell=0}^{J-M} (\ell+1) 2^{J+K-\ell} \\
 & \leq 2^{K+1}M2^{J+1} = 2^{K+1}M2^{M+k} \\
 & \leq 2^{K+1} \frac{M}{2^M} b  \leq \frac{b}{8}
\end{align*}
for $M$ sufficiently large.

Therefore, \eqref{lowboundeq1} gives us
$$|H(n,b)| \gg \frac{q^n}{b^2} \sum_{\mathbf{b}\in \mathcal{B}}\frac{\left(\sum_{A\in\A(\mathbf{b})} \frac{\tau(A)}{|A|}\right)^2}{\sum_{A\in \A(\mathbf{b})} \frac{W(A)}{|A|}}.$$

Now, if we let
$$f(\mathbf{b}) = \sum_{h=M}^J 2^{M-1-h+b_M+\cdots+b_h}$$
then we have by Lemma \ref{WAlem} that
\begin{align}\label{WAbound2}
\sum_{A\in \A(\mathbf{b})} \frac{W(A)}{|A|} \ll \frac{(2\log(2))^k}{b_M!\cdots b_J!} \left(1+2^{1-M}f(\mathbf{b})\right) \leq \frac{(2\log(2))^k}{b_M!\cdots b_J!}f(\mathbf{b})
\end{align}
since $f(\mathbf{b})\geq 1/2$. Hence, by Lemma \ref{tauAlem}, \eqref{lowboundeq1} and \eqref{WAbound2}, we get
$$|H(n,b)| \gg \frac{q^{n}(2\log(2))^k}{b^2} \sum_{\mathbf{b}\in\mathcal{B}} \frac{1}{b_M!\cdots b_J! f(\mathbf{b})}.$$

Finally, Ford in \cite{F2} shows that
$$\sum_{\mathbf{b}\in\mathcal{B}} \frac{1}{b_M!\cdots b_J! f(\mathbf{b})} \gg \frac{k^{k-1}}{k!} \gg \frac{1}{k^{3/2}},$$
where the last inequality is due to Stirling's formula.

Therefore, since $k \sim \log(b)/\log(2)$, we get
$$|H(n,b)| \gg \frac{q^{n}}{b^\delta \log(b)^{3/2}}$$
which finished the proof of the lower bound.

\section{Upper Bound}\label{Section4}

Before we begin, we need some basic bounds for $L(A)$.

\begin{lem}\label{Upboundlem1}

\begin{enumerate}

\item $L(A) \leq \min(\tau(A), \deg(A))$
\item If $(A,B)=1$, then $L(AB)\leq \tau(B)L(A)$
\item If $P_1,\dots,P_k$ are distinct primes, then $L(P_1\cdots P_k)\leq \min_{0\leq j \leq k} (2^{k-j} \deg(P_1\cdots P_j))$

\end{enumerate}

\end{lem}

\begin{proof}

For part $(1)$, we have
$$L(A) = \sum_{d\in\Ll(A)} 1  \leq \sum_{D|A} 1 = \tau(A).$$
While on the other hand, $\Ll(A) \subset \{1,\dots,\deg(A)\}$ and so $L(A) \leq \deg(A)$.

For part $(2)$, we have
$$\Ll(AB) = \bigcup_{D|B} \{d+\deg(D) : d\in \Ll(A)\}$$
and so $L(AB) \leq \sum_{D|B} L(A) = \tau(B)L(A)$.

Part $(3)$ follows from applying parts $(1)$ and $(2)$ with $A = P_1\cdots P_j$ and $B=P_{j+1}\cdots P_k$.

\end{proof}

We shall first prove the upper bound in the case of squarefree polynomials. That is, let $H^*(n,b)$ be the set of squarefree polynomials in $M_n$ which has a divisor of degree $b$.

\begin{lem} \label{sqfreelem}
For $b\leq n/2$,
$$|H^*(n,b)| \ll q^n (S(b)+S(n-b)),$$
as $q^n\to\infty$, where
$$S(d) = \sum_{\substack{\deg(P^+(A))\leq d \\ \mu^2(A)=1}} \frac{L(A)}{|A|\left(\deg(P^+(A))+d-\deg(A)\right)^2}$$
and $P^+(A)$ denotes the largest prime divisor of $A$ and $\mu$ is the M\"obius function.
\end{lem}

\begin{proof}

Let $F\in H^*(n,b)$. Then $F=G_1G_2$ where $G_1\in M_b$ and $G_2\in M_{n-b}$. Moreover, necessarily, $G_1$ and $G_2$ are squarefree and coprime.

First, suppose that $\deg(P^+(G_1))\leq \deg(P^+(G_2))$ and choose $P|G_1$ such that $\deg(P)=\deg(P^+(G_1))$. Write $F=ABP$ such that $\deg(P^+(A))\leq \deg(P)$ and all primes dividing $G_1$, except for $P$, divide $A$ and $\deg(P^-(B))\geq \deg(P)$ and all primes dividing $G_2$ with degree greater than or equal to $P$ divides $B$.

Then, by design we have $AP$ has a divisor of degree $b$. Therefore, $\deg(P)\geq b-\deg(A)$. Moreover, if we fix $A$ and $P$, we get that $B\in \M_{n-\deg(AP)}$ with $\deg(P^-(B))\geq \deg(P)$. Therefore, by \eqref{smooth} the number of such $B$ will be
$$\ll \frac{q^n}{|AP|\deg(P)}$$

We know that $A$ has a divisor of degree $b-\deg(P)$. So we get that
\begin{align} \label{sqfreelemeq1}
\sum_{\substack{\deg(P)\geq C \\ b-\deg(P)\in L(A)}} \frac{1}{|P|\deg(P)} &\ll \frac{1}{C} \sum_{\substack{d\in\Ll(A) \\ d-b\geq C }} \sum_{P\in M_{d-b}} \frac{1}{|P|}\\
& \ll \frac{1}{C} \sum_{\substack{d\in\Ll(A) \\ d-b\geq C }} \frac{1}{d-b} \ll \frac{L(A)}{C^2} \nonumber
\end{align}

We have that $\deg(P) \geq \max(\deg(P^+(A)),b-\deg(A))$. The case where $\deg(P^+(A))\leq b-\deg(A)$ will contribute to $H^*(n,b)$ at most
\begin{align*}
& q^n \sum_{\substack{ \deg(P^+(A)) \leq b-\deg(A) \\ \mu^2(A)=1}} \frac{1}{|A|} \sum_{\substack{\deg(P)\geq b-\deg(A) \\ b-\deg(P)\in \Ll(A)}} \frac{1}{|P|\deg(P)} \\
 \ll & q^n \sum_{\substack{\deg(P^+(A))\leq b \\ \mu^2(A)=1}} \frac{L(A)}{|A|(b-\deg(A))^2} \\
 \ll & q^n S(b),
\end{align*}
where the last inequality comes from the fact that $b-\deg(A) \geq (\deg(P^+(A))+b-\deg(A))/2$ in this case.

In the case where $\deg(P^+(A))\geq d-\deg(A)$, then $\deg(P) \geq \deg(P^+(A))$. Moreover, since $AP$ has a divisor of degree $b$, we must have $\deg(P^+(A))\leq b$. Hence we get this case contributes to $H^*(n,b)$ at most
\begin{align*}
&q^n \sum_{\substack{b-\deg(A)\leq \deg(P^+(A))\leq b \\ \mu^2(A)=1}} \frac{1}{|A|} \sum_{\substack{\deg(P) \geq \deg(P^+(A)) \\ b-\deg(P) \in \Ll(A)}} \frac{1}{|P|\deg(P)}\\
\ll &  q^n \sum_{\substack{ \deg(P^+(A))\leq b \\ \mu^2(A)=1}} \frac{L(A)}{|A|\deg(P^+(A))^2}\\
\ll & q^n S(b),
\end{align*}
where again the last inequality comes from the fact that $\deg(P^+(A)) \geq (\deg(P^+(A))+b-\deg(A))/2$ in this case.

Therefore, we get a contribution of at most $q^nS(b)$ under the assumption that $\deg(P^+(G_1))\leq \deg(P^+(G_2))$. Now, suppose $F=G_1G_2$ with $G_1\in M_b$, $G_2\in M_{n-b}$ such that $\deg(P^+(G_2))\leq \deg(P^+(G_1))$ and choose $P|G_2$ such that $\deg(P)=\deg(P^+(G_2))$. Then write $F=ABP$ such that $\deg(P^+(A))\leq \deg(P)$, all primes that divide $G_2$ divide $A$ and $\deg(P^-(B))\geq \deg(P)$ and all the primes dividing $G_1$ whose degree is greater than or equal to $\deg(P)$ divide $B$.

Following the same logic as above with $b$ replaced with $n-b$, we get that this contributes at most
$$\ll q^n \sum_{\substack{\deg(P^+(A))\leq n-b \\ \mu^2(A)=1}} \frac{L(A)}{|A|\left(n-b-\deg(A)+\deg(P^+(A))\right)^2} = q^nS(n-b)$$
which concludes the proof.

\end{proof}

Define
$$T(d,m) = \sum_{\substack{\deg(P^+(A))\leq d \\ \deg(A)\geq m, \mu^2(A)=1}} \frac{L(A)}{|A|}$$
If either $\deg(A)\leq d/2$ or $\deg(P^+(A))\geq \epsilon d$, then $(d-\deg(A)+\deg(P^+(A)))^2 \gg d^2$. Conversely if $\deg(P^+(A))\leq \epsilon d$ then we can find a $0\leq g \leq \log(d)+\log(\epsilon)$ such that $e^g \leq \deg(P^+(A))\leq e^{g+1}$ and we get
\begin{align*}
S(d) & = \sum_{\substack{\deg(P^+(A))\leq d \\ \mu^2(A)=1}} \frac{L(A)}{|A|\left(\deg(P^+(A))+d-\deg(A)+\right)^2} \\
&\ll \frac{T(d,1)}{d^2} + \sum_{\substack{\deg(P^+(A))\leq \epsilon d \\ \deg(A)\geq d/2, \mu^2(A)=1}} \frac{L(A)}{|A|\left(\deg(P^+(A))+d-\deg(A)\right)^2}\\
&\ll \frac{T(d,1)}{d^2} + \sum_{g=0}^{\log(d)+\log(\epsilon)} \sum_{\substack{e^{g-1}\leq \deg(P^+(A))\leq e^g \\ \deg(A)\geq d/2, \mu^2(A)=1}} \frac{L(A)}{|A|\left(\deg(P^+(A))+d-\deg(A)\right)^2}\\
&\ll \frac{T(d,1)}{d^2} + \sum_{g=0}^{\log(d)+\log(\epsilon)} e^{-2g}T(e^g,d/2).
\end{align*}

Finally define
$$T_k(d,m) = \sum_{\substack{\deg(P^+(A))\leq d \\ \deg(A)\geq m, \mu^2(A)=1 \\ \omega(A)=k}} \frac{L(A)}{|A|},$$
where $\omega(A)$ is the number of prime divisors of $A$.

\begin{lem}\label{Tklem}

For $d$ large and $m\geq 1$, let $v=\lfloor \log_2(d)\rfloor$. The for $1\leq k \leq 10v$, we have
$$T_k(d,m) \ll e^{-m/d} (2\log(d))^k \frac{1+|v-k|^2}{(k+1)!(2^{k-v}+1)}$$

\end{lem}

\begin{proof}
Firstly,
$$T_k(d,m) \leq \sum_{\substack{\deg(P^+(A))\leq d \\ \deg(A)\geq m, \omega(A)=k}} \frac{L(A)}{|A|} \leq e^{-m/d} \sum_{\substack{\deg(P^+(A))\leq d \\  \omega(A)=k}} \frac{L(A)}{|A|^{1-1/\log(q) d}}$$
Now, by \eqref{inverseprime2} we get
$$\sum_{\deg(P)\leq d} \frac{1}{|P|^{1-1/\log(q)d}}  = \log(d) + O(1).$$
Therefore, we can partition the interval $[1,d]$ into subintervals $E_0, \dots, E_{v+K-1}$ (for some constant $K$) such that for all $j$, $E_j$ is the next largest interval such that
$$\sum_{\substack{P\in M_e\\e\in E_j}} \frac{1}{|P|^{1-1/\log(q) d}} \leq \log(2)$$
Consequently, $P\in E_j$ implies that $\deg(P)\leq 2^{j+K}$.

Now, let $A=P_1\cdots P_k$ with $\deg(P_1)\leq \dots \leq \deg(P_k)\leq d$. Let $j_i$ be such that $P_i\in E_{j_i}$. Then Lemma \ref{Upboundlem1} says
$$L(A) \leq \min_{0\leq t \leq k} 2^{k-t} \deg(P_1\cdots P_t) \leq 2^{k+K} \min_{0\leq t \leq k} 2^{-t} \sum_{i=1}^t 2^{j_i}$$
Therefore, if we define
$$F(\mathbf{j}) := \min_{0\leq t \leq k} 2^{-t} \sum_{i=1}^t 2^{j_i}$$
then
$$T_k(d,m) \leq q^{-m/d}  2^{k+K} \sum_{\mathbf{j}\in J} F(\mathbf{j}) \sum_{\substack{P_1,\dots,P_k  \\ P_i\in E_{j_i}}} \frac{1}{|P_1\dots P_k|^{1-1/\log(q) d}} $$
where $J$ is the set of all vectors $\mathbf{j}$ such that $j_1\leq \dots \leq j_k \leq v+K-1$.

Fix a $\mathbf{j} = (j_1,\dots, j_k)$. For each $0\leq j\leq v+K+1$, let $b_j$ be the number of $i$ such that $j_i=j$. Then the inner sum of $P_1,\dots,P_k$ will be less than
\begin{align*}
\prod_{j=1}^{v+K-1} \frac{1}{b_j!} \left(\sum_{P\in E_j} \frac{1}{|P|^{1-1/\log(q) d}}\right)^{b_j} & \leq  \frac{\log(2)^k}{b_0!\cdots b_{v+K-1}!}\\
& = ((v+K)\log(2))^k \int_{R(\mathbf{j})} 1 d\mathbf{\xi}\\
& \leq e^{10K} (v\log(2))^k \int_{R(\mathbf{j})} 1 d\mathbf{\xi}
\end{align*}
where
$$R(\mathbf{j}) = \{0\leq \xi_1\leq \dots \leq \xi_k \leq 1: j_i \leq (v+K)\xi_i \leq j_i+1 \forall i\}$$
and the last inequality uses the hypothesis that $k \leq 10v$.

Finally, Ford in \cite{F2} shows that
$$\sum_{\mathbf{j}\in J} F(\mathbf{j}) \int_{R(\mathbf{j})} 1 d\mathbf{\xi} \ll \frac{1+|v-k|^2}{(k+1)!(2^{k-v}+1)}$$
and the lemma follows.
\end{proof}

\begin{lem}\label{Tlem}

$$T(d,m) \ll e^{-m/d} \frac{d^{2-\delta}}{\log(d)^{3/2}}$$

\end{lem}

\begin{proof}

We clearly have
$$T(d,m) = \sum_k T_k(d,m)$$

Then if $v=\lfloor \log_2(d) \rfloor$, Lemma \ref{Tklem} says that
$$\sum_{v \leq k \leq 10v} T_k(d,m) \ll e^{-m/d} \sum_{v \leq k \leq 10v} \frac{1+(k-v)^2}{2^{k-v}} \frac{(2\log(d))^k}{(k+1)!} \ll e^{-m/d} \frac{(2\log(d))^v}{(v+1)!}.$$
For $1\leq k \leq v$, we have
\begin{align*}
\sum_{1 \leq k \leq v} T_k(d,m) & \ll 2^vq^{-m/d} \sum_{1 \leq k \leq v}  \frac{(1+(v-k)^2)(\log(d))^k}{(k+1)!} \\
& = e^{-m/d}(2\log(d))^v\sum_{0\leq k \leq v-1} \frac{1+k^2}{\log(d)^k (v-k+1)!} \\
& \ll e^{-m/d} \frac{(2\log(d))^v}{(v+1)!} \sum_{0\leq k \leq v-1} (1+k^2) \left(\frac{v+1}{\log(d)}\right) \cdots \left(\frac{v-k+1}{\log(d)}\right) \\
& \ll e^{-m/d} \frac{(2\log(d))^v}{(v+1)!} \sum_{0 \leq k \leq v-1}\frac{1+k^2}{\log(2)^k} \\
& \ll  e^{-m/d} \frac{(2\log(d))^v}{(v+1)!},
\end{align*}
where the second last inequality comes from the fact that $v-j \leq \log_2(d)$ for all $j$.

For $k\geq 10v$, we use the Lemma \ref{Upboundlem1} and the definition of $T_k(d,m)$ to get
\begin{align*} \sum_{k\geq 10v} T_k(d,m) & = \sum_{k \geq 10v} \sum_{\substack{\deg(P^+(A))\leq d \\ \deg(A)\geq m, \mu^2(A)=1 \\ \omega(A)=k}} \frac{L(A)}{|A|} \ll e^{-m/d}\sum_{k \geq 10v} 2^k \sum_{\substack{\deg(P^+(A))\leq d \\ \omega(A)=k, \mu^2(A)=1}} \frac{1}{|A|^{1-1/d}}\\
&  \ll e^{-m/d}\sum_{k \geq 10v} \frac{2^k}{k!} \left(\sum_{\deg(P)\leq d} \frac{1}{|P|^{1-1/d}}\right)^k \ll e^{-m/d} \sum_{k \geq 10v} \frac{2^k}{k!} \left(\log(d)+O(1)\right)^k\\
& \ll e^{-m/d} \frac{(2\log(d))^{10v}}{(10v)!} \ll e^{-m/d} \frac{(2\log(d))^{v}}{(v+1)!}
\end{align*}

Finally, we using Stirling's bound we get the desired result.

\end{proof}

Hence,
\begin{align*}
S(d) & \ll \frac{T(d,1)}{d^2} + \sum_{g=1}^{\log(\epsilon d)} e^{-2g}T(e^g,d/2) \\
& \ll \frac{q^{-1/d}}{d^{\delta}(\log(d))^{3/2}} + \sum_{g=1}^{\log(\epsilon d)} \frac{1}{q^{d/2e^g}e^{\delta g} g^{3/2} }\\
& \ll \frac{1}{d^\delta(\log(d))^{3/2}}
\end{align*}
and as long as we assume that $b\leq n/2$, then
\begin{align*}
|H^*(n,b)| & \ll q^n(S(b)+S(n-b))\\
& \ll q^n \left(\frac{1}{b^\delta(\log(b))^{3/2}} + \frac{1}{(n-b)^\delta(\log(n-b))^{3/2}}\right) \\
& \ll \frac{q^n}{b^\delta(\log(b))^{3/2}}
\end{align*}

It remains now to deduce the correct upper bound from the square-free case.

\begin{lem}
$$|H(n,b)| \ll \frac{q^n}{b^\delta(\log(b))^{3/2}}.$$
\end{lem}

\begin{proof}

Write $F=F'F''$ where $F'$ is square-free, $F''$ is square-full and $(F',F'')=1$. The number of $F$ with $\deg(F'')\geq (4+\epsilon)\log(b)$ will be less than
$$ q^n \sum_{\substack{F'' \mbox{ square-full} \\ deg(F'')\geq (4+\epsilon)\log(b)}}\frac{1}{|F''|} \ll \frac{q^n}{b^2}$$
by \eqref{sqfull1}

Now, suppose $\deg(F'')\leq (4+\epsilon)\log(b)$, then there is a $D|F''$ such that $F'$ has a divisor of degree $b-\deg(D)$. Thus
\begin{align*}
|H(n,b)| & \leq \sum_{\substack{F'' \mbox{ square-full} \\ \deg(F'')\leq (4+\epsilon)\log(b)}} \sum_{D|F''} |H^*(n-\deg(F''), b-\deg(D))| + O\left(\frac{q^n}{b^2}\right) \\
& \ll q^n \sum_{\substack{F'' \mbox{ square-full} \\ \deg(F'')\leq (4+\epsilon)\log(b)}} \sum_{D|F''} \frac{1}{|F''| (b-\deg(D))^{\delta} (\log(b-\deg(D)))^{3/2}} +  O\left(\frac{q^n}{b^2}\right)\\
& \ll \frac{q^n}{b^\delta(\log(b))^{3/2}} \sum_{\substack{F'' \mbox{ square-full} \\ \deg(F'')\leq (4+\epsilon)\log(b)}} \frac{\tau(F'')}{|F''|} +  O\left(\frac{q^n}{b^2}\right)\\
& \ll\frac{q^n}{b^\delta(\log(b))^{3/2}},
\end{align*}
where the last inequality is due to \eqref{sqfull2}.

\end{proof}

\appendix

\section{Estimates on Polynomials}\label{AppA}

In the whole appendix we will frequently use the prime polynomial theorem:
$$\pi_q(n) := |\{P\in \M_n : P \mbox{ is a prime polynomial}\}| = \frac{q^n}{n} + O\left(\frac{q^{n/2}}{n}\right).$$

\subsection{Rough Polynomials}

In this section we prove the following result:
\begin{align}\label{smooth}
|\{F\in \M_n : \deg(P^-(F))\geq d\}| \asymp \frac{q^n}{d}, \quad \quad (d\leq n)
\end{align}
as $q^n\to\infty$ where $P^-(F)$ denotes the smallest prime divisor of $F$.

Consider the generating series
$$\sum_{\substack{F \\ \deg(P^-(F)\geq d}} \frac{1}{|F|^s} = \prod_{\deg(P)\geq d} \left(1-\frac{1}{|P|^s}\right)^{-1} = \zeta_q(s) \prod_{\deg(P)<d} \left(1-\frac{1}{|P|^s}\right).$$
Hence, standard analytic tools show that
$$\sum_{\substack{F\in \M_n \\ \deg(P^-(F))\geq d}} 1 = q^n \prod_{\deg(P)<d}\left(1-\frac{1}{|P|}\right) + O\left(q^{(1/2+\epsilon)n}\right).$$
Finally,
\begin{align*}
\log \prod_{\deg(P)< d} \left(1-\frac{1}{|P|}\right)& = \sum_{\deg(P)< d} \log \left(1-\frac{1}{|P|}\right) = -\sum_{k=1}^{\infty} \sum_{\deg(P)< d} \frac{1}{|P|^k} \\
& = -\sum_{k=1}^{\infty} \sum_{m\leq d} \frac{\pi_q(m)}{q^{mk}} \\
& = -\sum_{m\leq d} \frac{1}{m} \sum_{k=0}^{\infty} \frac{1}{q^{km}} + O\left(\sum_{m\leq d} \frac{q^{m/2}}{m} \sum_{k=1}^{\infty} \frac{1}{q^{mk}} \right)\\
& = -\sum_{m\leq d} \frac{1}{m} - \sum_{m\leq d} \frac{1}{m(q^m-1)} + O\left(\sum_{m\leq d} \frac{1}{mq^{m/2}}\right)\\
& = \log(1/d) + O(1)
\end{align*}
where the constant in the term $O(1)$ is independent of $q$.

\subsection{Sum of Inverse Prime Polynomials}

In this section we prove that
\begin{align} \label{inverseprime}
\sum_{d_1\leq\deg(P)\leq d_2} \frac{1}{|P|} = \log(d_2)-\log(d_1) + O\left(\frac{1}{d_1}\right).
\end{align}
and
\begin{align} \label{inverseprime2}
\sum_{\deg(P)\leq d} \frac{1}{|P|^{1-1/\log(q)d}} = \log(d)+O(1).
\end{align}
where the implied constants are independent of $q$

Applying the prime polynomial theorem, we get
\begin{align*}
\sum_{d_1\leq\deg(P)\leq d_2} \frac{1}{|P|} = \sum_{n=d_1}^{d_2} \frac{\pi_q(n)}{q^n} = \sum_{n=d_1}^{d_2} \left(\frac{1}{n} + O\left(\frac{1}{nq^{n/2}}\right)\right) = \log(d_2)-\log(d_1) + O\left(\frac{1}{d_1}\right).
\end{align*}

Further, since $\deg(P)\leq d$, we get that $|P|^{1/d\log(q)} = e^{\deg(P)/d} = 1+O\left(\deg(P)/d\right)$. Hence,
\begin{align*}
\sum_{\deg(P)\leq d} \frac{1}{|P|^{1-1/d\log(q)}} & = \sum_{\deg(P)\leq d} \frac{1}{|P|} + O\left(\frac{1}{d}\sum_{\deg(P)\leq d} \frac{\deg(P)}{|P|}\right) \\
 & = \log(d)+O(1) + O\left(\frac{1}{d} \sum_{n\leq d} \frac{n\pi_q(n)}{q^n}\right)\\
 & = \log(d)+O(1).
\end{align*}

\subsection{Sum of Squarefull Polynomials}

In this section we prove that
\begin{align}\label{sqfull1}
\sum_{\substack{F \mbox{ square-full} \\ \deg(F)\geq C}} \frac{1}{|F|} \ll q^{-(1/2-\epsilon)C}
\end{align}
and
\begin{align}\label{sqfull2}
\sum_{F \mbox{ square-full}} \frac{\tau(F)}{|F|} = O(1),
\end{align}
where the implied constants are independent of $q$.

We have the identity
\begin{align*}
\sum_{F \mbox{ square-full}} \frac{1}{|F|^s} & = \prod_{P} \left(1 + \frac{1}{|P|^{2s}}+\frac{1}{|P|^{3s}}+\cdots \right)\\
& = \prod_{P} \left(\frac{1-1/|P|^{6s}}{(1-1/|P|^{2s})(1-1/|P|^{3s})}\right)\\
& = \frac{\zeta_q(2s)\zeta_q(3s)}{\zeta_q(6s)}.
\end{align*}
So the generating series can be analytically continued to $\Re(s)>1/2$. Hence we have
$$|\{F \mbox{ square-full} : \deg(F)=n \}| \ll q^{(1/2+\epsilon)n}$$
and therefore
\begin{align*}
\sum_{\substack{F \mbox{ square-full} \\ \deg(F)\geq C}} \frac{1}{|F|} & = \sum_{n\geq C} \frac{|\{ F \mbox{ square-full} : \deg(F)=n  \}|}{q^n} \\
& \ll \sum_{n\geq C} q^{-(1/2-\epsilon)n} \ll q^{-(1/2-\epsilon)C}.
\end{align*}

Finally,
$$\sum_{F \mbox{ square-full}} \frac{\tau(F)}{|F|^s}  = \prod_{P} \left(1 + \frac{3}{|P|^{2s}}+\frac{4}{|P|^{3s}}+\cdots \right)$$
and so converges at $s=1$ and tends to $1$ in the $q$-limit.

\bibliography{Multfirstdraft}
\bibliographystyle{amsplain}

\end{document}